\documentclass[12pt]{article}
\usepackage{amsmath, amsthm, latexsym, amssymb,url}
\usepackage{amsfonts}
\usepackage{epsfig}
\usepackage{titlesec}
\titleformat{\section}{\normalfont\Large\bfseries}{\S\thesection}{1em}{}

\usepackage{graphicx}
\usepackage{yhmath}
\usepackage{mathdots}

\newcommand{\shrinkmargins}[1]{
  \addtolength{\textheight}{#1\topmargin}
  \addtolength{\textheight}{#1\topmargin}
  \addtolength{\textwidth}{#1\oddsidemargin}
  \addtolength{\textwidth}{#1\evensidemargin}
  \addtolength{\topmargin}{-#1\topmargin}
  \addtolength{\oddsidemargin}{-#1\oddsidemargin}
 \addtolength{\evensidemargin}{-#1\evensidemargin}
  }

\shrinkmargins{.2}

\theoremstyle{plain}

\newtheorem{theorem}{Theorem}[section]

\newtheorem{lemma}[theorem]{Lemma}

\theoremstyle{remark}

\theoremstyle{definition}

\def \Q { \mathbb{Q}}

\def \det { \text{det}}

\newcommand{\GCD}{\text{GCD}}
\newcommand{\LCM}{\text{LCM}}

\begin{document}

\thispagestyle{empty}

\title{An $\ell-p$ switch trick to obtain a new proof of a criterion for arithmetic equivalence.}
\author{Tristram Bogart \and  Guillermo Mantilla-Soler}
\date{}

\maketitle

\begin{abstract}
Two number fields are called {\it arithmetically equivalent} if they have the same Dedekind zeta function. In the 1970's Perlis showed that this is equivalent to the condition that for almost every rational prime $\ell$ the arithmetic type of $\ell$ is the same in each field. In the 1990's Perlis and Stuart gave an unexpected characterization for arithmetic equivalence; they showed that to be arithmetically equivalent it is enough for almost every prime $\ell$ to have the same number of prime factors in each field. Here, using an $\ell-p$ switch trick, we provide an alternative proof of that fact based on a classical elementary result of Smith from the 1870's.
\end{abstract}

\section{Introduction}
One of the most fundamental invariants of a number field is its Dedekind
zeta function \cite[Chapter VII, \S 5]{Neukirch}. It is well known that pairs of number fields with the same zeta function share many arithmetic invariants, including the discriminant, unit group, signature, the product of class number times regulator and others (see \cite{Perlis}, \cite{Klingen} and \cite{Manti2}.) A first characterization for equality between zeta functions was given by Perlis in terms of residue class degrees of rational primes. Let $K$ be a number field with maximal order $O_{K}$. The {\it arithmetic type} of a rational prime $\ell$ in $K$ is the ordered tuple $A_{K}(\ell):=(f_{1}, \ldots, f_{g})$ where $f_{1} \leq \cdots \leq f_{g}$ are the residue class degrees of $\ell$ in $K$. Perlis' characterization of the zeta function is the following:
\begin{theorem}[Perlis \cite{Perlis}]
Let $K$ and $L$ be number fields. The following are equivalent:

\begin{itemize}
\item[(a)] The fields $K$ and $L$ have the same zeta functions:  \[\zeta_{K}(s)=\zeta_{L}(s).\]

\item[(b)] For almost every rational prime $\ell$ the arithmetic types of $\ell$ in K and L coincide: \[A_{K}(\ell)=A_{L}(\ell).\]

\end{itemize}
\end{theorem}

Here and in Theorem \ref{ElTeorema}, "almost all" means "with the possible exception of a set of primes of Dirichlet density zero."

Given a rational prime $\ell$ the number $g=g_{K}(\ell)$ of prime factors of $\ell$ in $O_{K}$ is called the {\it splitting number} of $\ell$ in $K$. It is obvious that if two number fields have the same arithmetic type for every prime $\ell$ then they have the same splitting number for every prime $\ell$. The following result, due to Perlis and Stuart, shows that the converse also holds:

\begin{theorem}\cite{PerlisStuart}\label{ElTeorema}
Let $K$ and $L$ be number fields. Then, the following are equivalent:
\begin{itemize}
\item[(a)] For almost all rational primes $\ell$ the arithmetic types of $\ell$ in $K$ and $L$ coincide: \[A_{K}(\ell)=A_{L}(\ell).\]

\item[(b)] For almost all rational primes $\ell$ the splitting numbers of $\ell$ in $K$ and $L$ coincide: \[g_{K}(\ell)=g_{L}(\ell).\]
\end{itemize}

\end{theorem}

A known proof of Theorem \ref{ElTeorema} is obtained by using a result (\cite[Theorem 30']{Serre}) on virtual characters and rational representations. See \cite[III, \S1, Theorem 1.5]{Klingen} or the original \cite{PerlisStuart} for details. The goal of this article is to give a short proof of Theorem \ref{ElTeorema} based on a classical result of H.J.S. Smith involving matrices of greatest common divisors. To apply Smith's result we use an $\ell$-$p$ switch trick; we begin by focusing on a Frobenius element for a prime $\ell$ and then switch to a Frobenius element for a different prime $p$.

\section{Proofs}
In this section we develop our new proof of Theorem \ref{ElTeorema}. We begin with an elementary observation in linear algebra.   
\begin{lemma}\label{LosValoresPropios}
Let $A$ be an $n \times n$ matrix with complex entries, and let $f_1,\ldots,f_g$ be positive integers. Suppose that the characteristic polynomial of $A$ is \[{\rm det}(XI-A) =\prod_{i=1}^{g}\left(X^{f_{i}}-1\right).\] Then for every positive integer $N$, the algebraic multiplicity of the eigenvalue $\lambda=1$ in the matrix $A^{N}$ is equal to 

\[ \sum_{d | N}  \# \{f_i:  {\rm GCD}(f_i, N) = d \}d\]
\end{lemma}

\begin{proof}
Let $\mu_{1}(\cdot)$ denote the function that calculates the algebraic multiplicity of the eigenvalue $1$ of a given linear transformation. First we show that we may assume that $A$ is diagonal. The matrix $A$ can be written, up to conjugacy, as \[A=D+M\] where $M$ is nilpotent upper triangular and $D$ is  diagonal  such that $DM=MD$. Since $A$ and $D$ have the same characteristic polynomial it follows that for any common eigenvalue $\lambda$  the algebraic multiplicity of it in  $A$ is the same as it is in $D$.  It follows from this, and from the fact  that for all positive integer $N$ \[(D+M)^N=D^N + M'\] where $M'$ is an upper triangular nilpotent matrix, that we may assume that $A$ is a diagonal matrix. More succinctly we have shown that for all positive integer $N$  \[ \mu_{1}(A^N)=\mu_{1}(D^N).\]

\noindent Since $A$ is assumed to be diagonal it can be written as a direct sum of diagonal matrices $A_{i}$ where each $A_{i}$ has characteristic polynomial $\displaystyle X^{f_{i}}-1$. In particular, for all positive integers $N$ we have \[ \mu_{1}(A^{N})=\sum_{i=1}^{g}\mu_{1}(A_{i}^{N}).\] Notice that the eigenvalues of $A_{i}$ are exactly the elements of the group of $f_i$-roots of unity. Hence, $\mu_{1}(A_{i}^{N})$ is the number of $f_i$-roots whose orders divide $N$. Since that number is ${\rm GCD}(f_{i}, N)$, we conclude that \[\mu_{1}(A^{N})=\sum_{i=1}^{g} {\rm GCD}(f_{i}, N)=\sum_{d | N}  \# \{f_i:  {\rm GCD}(f_i, N) = d \}d. \]
\end{proof}

We will also need the following statement in elementary number theory. 

\begin{theorem} \label{thm:2sequences}
  Suppose $a_1 \leq a_2 \leq \ldots \leq a_m$ and $b_1 \leq b_2 \leq \ldots \leq b_n$ are two sequences of integers that satisfy the equation
  \begin{equation} \label{eq:hypothesis} \sum_{d | N} d \# \{a_k: \GCD(a_k, N) = d \} = \sum_{d | N} d \# \{b_k: \GCD(b_k, N) = d \} \end{equation}
  for every positive integer $N$. Then $m = n$ and $a_i = b_i$ for $i=1, \ldots, m$.  
\end{theorem}

To prove this, we will apply the following theorem of H.J.S. Smith.

\begin{theorem} \cite{Smith1875} \label{thm:Smith}
Let $S = \{ s_1, \ldots, s_c\}$ be a \emph{factor-closed} set of positive integers; that is, if $s \in S$, then every positive divisor of $s$ also belongs to $S$. Let $M = (m_{ij})$ be the symmetric $c \times c$ matrix given by $m_{ij} = \GCD(s_i, s_j)$. Then $\det(M) = \prod_{i=1}^c \phi(s_i)$.  
\end{theorem}

In a recent survey by Krattenthaler, it is observed that Smith's theorem is a consequence of a more general property of determinants involving M\"obius functions of posets \cite[Theorems 55 and 57]{Krattenthaler2005}.

\begin{proof}[Proof of Theorem \ref{thm:2sequences}]
  First, consider equation (\ref{eq:hypothesis}) with $N=1$. Then the only term in each sum comes from $d=1$ and we obtain that
  \[ \# \{a_k: \GCD(a_k, 1) = 1 \}  =  \# \{b_k: \GCD(b_k, 1) = 1 \} ; \]
  that is, $m=n$.
  \par
  Now let $L = \LCM \{a_1, \ldots, a_n, b_1, \ldots, b_n\}$ and $S := \{s_1, \ldots, s_c\}$ be the set of positive divisors of $L$. In particular, $S$ is factor-closed and $\{a_1, \ldots, a_n, b_1, \ldots, b_n\} \subseteq S$. For $j=1, \ldots, c$, also define
  \[ x_j := \# \{k: a_k = s_j \}  -  \# \{k: b_k = s_j \}  .\]
  \par
  For each $i=1, \ldots, c$, we now consider equation (\ref{eq:hypothesis}) for $N=s_i$. After moving everything to the left side, this equation becomes
  \[ \sum_{d | s_i} d \cdot \left( \# \{k: \GCD(a_k, s_i) = d \} - \# \{k: \GCD(b_k, s_i) = d \} \right) = 0.  \]
  Now since every $a_k$ and every $b_k$ belong to $S$, we can rewrite this as
  \[ \sum_{d | s_i} \sum_{\substack{j \\ \GCD(s_j, s_i) = d}} d \cdot \left( \# \{k: a_k = s_j \} - \# \{k: b_k = s_j \} \right) = 0.  \] 
  That is,
  \[ \sum_{d | s_i} \sum_{\substack{j \\ \GCD(s_j, s_i) = d}} d x_j = 0.\]
  Now each value of $j$ (from 1 to $c$) occurs for exactly one value of $d$, so we can eliminate $d$ and obtain
  \[ \sum_{j=1}^c \GCD(s_j, s_i) x_j = 0.\]
\par 
  By varying $i$, we create a system of $c$ linear equations in $c$ variables represented by the matrix equation
  \[ M \mathbf{x} = \mathbf{0} \]
  where $M = (m_{ij})$ is given by $m_{ij} = \GCD(s_j, s_i) = \GCD(s_i, s_j)$. By Theorem \ref{thm:Smith}, $\det(M) = \prod_{i=1}^c \phi(s_i)$. In particular, $M$ is nonsingular, so the only solution is the trivial one $x_1 = \ldots = x_c = 0.$ That is, for each $j$ the two increasing sequences $a_1 \leq \ldots \leq a_n$ and $b_1 \leq \ldots \leq b_n$ contain the same number of terms that are equal to $s_j$. But every term in both sequences is equal to some $s_j$, and therefore the two sequences coincide.   
\end{proof}

\begin{proof}[Proof of Theorem \ref{ElTeorema}]
It is enough to show that (b) implies (a). Recall that the Dedekind zeta function of a number field $E$ is equal to the Artin $L$-function \cite[Chapter VII, \S 10]{Neukirch} of the natural permutation representation $\rho_E$ given by the action of the absolute Galois group $G_\Q$ on the complex embeddings of $E$. Moreover, if $\ell$ is a prime unramified in $E$ then the characteristic polynomial of the image of its Frobenius element under $\rho_{E}$ is given by \[{\rm det}(XI-\rho_{E}({\rm Frob}_{\ell}))=\prod_{i=1}^{g}(X^{f_{i}}-1)\] where the $f_{i}$'s are the residue class degrees of $\ell$.  From this we see that $g_{E}(\ell)$ is equal to the algebraic multiplicity of the eigenvalue $1$ of the transformation $\displaystyle \rho_{E}({\rm Frob}_{\ell})$. For details see \cite[\S2]{Manti}, in particular Lemma 2.4.\\

\noindent Let $\ell$ be a prime unramified in both $K$ and $L$ and let $N$ be a positive integer. We now apply our $\ell-p$ switch. By Chebotarev's density theorem, the set of primes $p$ for which $\displaystyle {\rm Frob}^{N}_{\ell}={\rm Frob}_{p}$ has positive density. We can thus choose such a prime $p$ that is unramified in both fields. 
By hypothesis $g_{K}(p)=g_{L}(p)$, hence the algebraic multiplicities of the eigenvalue $1$ of $\rho_{K}({\rm Frob}_{p})$ and $\rho_{L}({\rm Frob}_{p})$ coincide. Therefore the algebraic multiplicities of the eigenvalue $1$ of $\rho_{K}({\rm Frob}_{\ell})^{N}$ and $\rho_{L}({\rm Frob}_{\ell})^{N}$ are equal. \\

\noindent Write $A_K(\ell) = (f_1, \ldots, f_g)$ and $A_L(\ell) = (f'_1, \ldots, f'_{g'})$. Applying Lemma \ref{LosValoresPropios} to both $\rho_{K}({\rm Frob}_{\ell})^{N}$ and $\rho_{L}({\rm Frob}_{\ell})^{N}$, we obtain that
\[\sum_{d | N}  \# \{f_i:  {\rm GCD}(f_i, N) = d \}d = \sum_{d | N}  \# \{f'_i:  {\rm GCD}(f'_i, N) = d \}d   \]
for every positive integer $N$. The result now follows from Theorem \ref{thm:2sequences}. \end{proof}

\section*{Acknowledgements}
We would like to thank the referee for the careful reading of the
paper, and specially for point it out a flaw in an argument of a previous version of the proof of Lemma \ref{LosValoresPropios}.

\noindent
{\footnotesize Tristram Bogart, Department of Mathematics, Universidad de los Andes,
Bogot\'a, Colombia ({\tt tcbogart@gmail.com})}

\noindent
{\footnotesize Guillermo Mantilla-Soler, Department of Mathematics, Universidad de los Andes,
Bogot\'a, Colombia ({\tt gmantelia@gmail.com})}

\end{document}